\documentclass[12pt]{amsart}

\usepackage[left=2.5cm,right=2.5cm,top=2.8cm,bottom=2.8cm]{geometry}
\usepackage{amsmath,amssymb,amsfonts,amsthm}
\usepackage{mathtools}
\usepackage[backref,colorlinks=true,linkcolor=blue,citecolor=blue,urlcolor=blue]{hyperref}

\newtheorem{theorem}{Theorem}[section]
\newtheorem{lemma}[theorem]{Lemma}

\newtheorem{problem}[theorem]{Problem}
\theoremstyle{definition}
\newtheorem{definition}[theorem]{Definition}
\theoremstyle{remark}
\newtheorem{remark}[theorem]{Remark}

\numberwithin{equation}{section}

\DeclareMathOperator{\Mat}{Mat}

\newcommand{\Q}{\mathbb Q}
\newcommand{\R}{\mathbb R}
\newcommand{\bfe}{\mathbf e}

\title[Edge-deleted strongly regular graphs]{Orthogonal degree-similarity of edge-deleted strongly regular graphs}

\author[Y.-Z. Fan]{Yi-Zheng Fan*}
\address{Center for Pure Mathematics, School of Mathematical Sciences, Anhui University, Hefei 230601, P. R. China}
\email{fanyz@ahu.edu.cn}
\thanks{*Supported by National Natural Science Foundation of China (Grant No. 12331012).}

\author[W. Wang]{Wei Wang$^\sharp$}
\address{School of Mathematics and Statistics, Xi'an Jiaotong University, Xi'an 710049, P. R. China}
\email{wang\_weiw@xjtu.edu.cn}
\thanks{$^\sharp$Corresponding author. Supported by National Natural Science Foundation of China (Grant No. 12371357).}

\author[K. Zhang]{Kuo Zhang}
\address{School of Mathematical Sciences, Anhui University, Hefei 230601, P. R. China}
\email{zhangk@stu.ahu.edu.cn}


\begin{document}

\begin{abstract}
Godsil and Sun asked whether, for a strongly regular graph $X$ and any two different edges $e$ and $f$, the edge-deleted graphs $X\setminus e$ and $X\setminus f$ are degree-similar.
We give an affirmative answer to the problem of Godsil and Sun.
In fact, we prove the stronger statement that if $X$ is a $1$-walk-regular graph, then for any two edges $e$ and $f$ of $X$, the graphs $X\setminus e$ and $X\setminus f$ are orthogonally degree-similar.
The proof is based on an edge version of the orthogonal-intertwiner method: the equality of the Gram matrices of the projected endpoint vectors in every eigenspace yields an orthogonal matrix commuting with the adjacency matrix and sending one pair of ordered endpoint vectors to the other.
\end{abstract}

\maketitle

\section{Introduction}

Throughout this paper all graphs are finite and simple.
For a graph $G$, let $A(G)$ denote its adjacency matrix and let $D(G)$ denote its degree matrix.  Following Godsil and Sun \cite{GodsilSun}, two graphs $G_1$ and $G_2$ are called \emph{degree-similar} if there exists an invertible matrix $M$ such that
\[
        M^{-1}A(G_1)M=A(G_2),
        \quad
        M^{-1}D(G_1)M=D(G_2).
\]
If such an $M$ can be chosen to be orthogonal, then we say that $G_1$ and $G_2$ are \emph{orthogonally degree-similar} \cite{FanXingZhangWang}.
Orthogonal degree-similarity is a stronger form of degree-similarity and implies simultaneous orthogonal similarity of several standard matrices associated with a graph, including the adjacency matrix, the Laplacian, the signless Laplacian and, when there are no isolated vertices, the normalized Laplacian.

The generalized $\mu$-adjacency matrix of a graph $G$ is
\[
        L_\mu(G):=A(G)-\mu D(G),
\]
and the corresponding $\mu$-polynomial is
\[
        \psi(G,t,\mu):=\det(tI-L_\mu(G)).
\]
If two graphs are degree-similar, then their generalized $\mu$-adjacency matrices are similar over $\Q(\mu)$, and hence the two graphs have the same $\mu$-polynomial.
Wang et al. \cite{WangLiLuXu} asked whether two graphs with same $\mu$-polynomial are degree-similar via an orthogonal matrix $M$.
Godsil and Sun \cite{GodsilSun} gave a negative answer to the problem by constructing pairs of graphs that share the common $\mu$-polynomials but are not degree-similar.
In \cite{GodsilSun}, Godsil and Sun posed several problems concerning degree-similar graphs.
The problem relevant here is the following.

\begin{problem}[Godsil-Sun \cite{GodsilSun}, Problem 10.3]\label{conj:GS}
Let $X$ be a strongly regular graph, and let $e$ and $f$ be two different edges of $X$.  Are $X\setminus e$ and $X\setminus f$ degree-similar?
\end{problem}

The motivation for this problem comes from a result of Godsil, Sun, and Zhang \cite{GodsilSunZhang}: if $X$ is strongly regular, then for any two edges $e$ and $f$, the graphs $X\setminus e$ and $X\setminus f$ are cospectral with respect to the adjacency matrix, the Laplacian, the signless Laplacian and the normalized Laplacian.
Fan, Xing, Zhang, and Wang \cite{FanXingZhangWang} later strengthened the polynomial side of this phenomenon by proving that, for a strongly regular graph $X$, the graphs $X\setminus e$ and $X\setminus f$ have the same $\mu$-polynomial.

In this paper we give an affirmative answer to Problem ~\ref{conj:GS}.
In fact, our result applies to the larger class of $1$-walk-regular graphs.

\begin{theorem}\label{thm:intro-main}
Let $X$ be a $1$-walk-regular graph.  Then for any two edges $e$ and $f$ of $X$, the graphs $X\setminus e$ and $X\setminus f$ are orthogonally degree-similar.
\end{theorem}

Since every strongly regular graph is $1$-walk-regular, Theorem~\ref{thm:intro-main} immediately gives a positive answer to Problem ~\ref{conj:GS}.
The proof uses an edge version of the orthogonal intertwiner, which was used in the study of cospectral vertices \cite{FanZhangWang}.
For vertices, one constructs an orthogonal map on each eigenspace by comparing the projected coordinate vectors.
For edges, one compares the two ordered pairs of projected endpoint vectors.
The $1$-walk-regular condition guarantees that these pairs have the same Gram matrix in every eigenspace, which yields an orthogonal matrix commuting with the adjacency matrix and sending one ordered edge to the other.

\section{Main result}

For a graph $G$ and a vertex $v\in V(G)$, or more generally for a square matrix $M$ and an index $v$, let $\bfe_v$ denote the standard basis column vector indexed by $V(G)$, or by the index set of $M$, respectively.
Thus $\bfe_v$ has its only nonzero entry, equal to $1$, in position $v$.

For two matrices $B$ and $C$ of the same size, let $B\circ C$ denote their Schur product (or Hadamard product), that is,
\[
        (B\circ C)_{u,v}=B_{u,v}C_{u,v}.
\]

\begin{definition}
Let $X$ be a graph with adjacency matrix $A$.
The graph $X$ is called \emph{$1$-walk-regular} if, for every integer $k\geq 0$, there exist scalars $a_k$ and $b_k$ such that
\[
        A^k\circ I=a_k I,
        \qquad
        A^k\circ A=b_k A.
\]
\end{definition}

This is equivalent to the usual definition: the number of closed walks of length $k$ is independent of the vertex, and the number of walks of length $k$ from $u$ to $v$ is independent of the adjacent pair $u,v$.
We shall use the following spectral characterization, which is standard; see, for example, \cite[Theorem 2.1]{GodsilSunZhang}.

\begin{lemma}[\cite{GodsilSunZhang}, Theorem 2.1]\label{lem:spectral-1wr}
Let $X$ be a $1$-walk-regular graph with adjacency matrix $A$, and let
\[
        A=\sum_{\theta\in\Theta}\theta E_\theta
\]
be its spectral decomposition.  Then, for each $\theta\in\Theta$, there exist scalars $\alpha_\theta$ and $\beta_\theta$ such that
\[
        E_\theta\circ I=\alpha_\theta I,
        \quad
        E_\theta\circ A=\beta_\theta A.
\]
Equivalently, the diagonal entries of $E_\theta$ are constant, and the entries of $E_\theta$ on adjacent pairs are constant.
\end{lemma}


\begin{remark}
Every strongly regular graph is $1$-walk-regular.  Indeed, if $X$ is strongly regular with parameters $(n,d;a,c)$ and adjacency matrix $A$, then
\[
        A^2=dI+aA+c(J-I-A).
\]
It follows that the algebra generated by $A$ is contained in the span of $I,A,J$.  Therefore every power $A^k$ has constant diagonal entries and constant entries on adjacent pairs.
\end{remark}


We first prove a linear-algebraic lemma.  It is the edge analogue of the usual orthogonal intertwiner associated with cospectral vertices \cite{FanZhangWang}.

\begin{lemma}[Edge intertwiner]\label{lem:edge-intertwiner}
Let $M\in\Mat_n(\R)$ be a symmetric matrix, and let
\[
        M=\sum_{\theta\in\Theta}\theta E_\theta
\]
be its spectral decomposition.  Let $(a,b)$ and $(c,d)$ be two ordered pairs of indices.  Suppose that, for every $\theta\in\Theta$, the two pairs of vectors
\[
        (E_\theta\bfe_a, E_\theta\bfe_b)
        \quad\text{and}\quad
        (E_\theta\bfe_c, E_\theta\bfe_d)
\]
have the same Gram matrix.  Equivalently,
\[
\begin{pmatrix}
\bfe_a^\top E_\theta\bfe_a & \bfe_a^\top E_\theta\bfe_b\\
\bfe_b^\top E_\theta\bfe_a & \bfe_b^\top E_\theta\bfe_b
\end{pmatrix}
=
\begin{pmatrix}
\bfe_c^\top E_\theta\bfe_c & \bfe_c^\top E_\theta\bfe_d\\
\bfe_d^\top E_\theta\bfe_c & \bfe_d^\top E_\theta\bfe_d
\end{pmatrix}
\]
for every $\theta\in\Theta$.
Then there exists an orthogonal matrix $Q\in\Mat_n(\R)$ such that
\[
        QM=MQ, \quad Q\bfe_a=\bfe_c, \quad Q\bfe_b=\bfe_d.
\]
\end{lemma}

\begin{proof}
Fix $\theta\in\Theta$, and put
\[
        x_1=E_\theta\bfe_a,
        \quad
        x_2=E_\theta\bfe_b,
        \quad
        y_1=E_\theta\bfe_c,
        \quad
        y_2=E_\theta\bfe_d.
\]
By assumption, the two ordered pairs $(x_1,x_2)$ and $(y_1,y_2)$ have the same Gram matrix.
Let
\[
        W_\theta=\operatorname{span}\{x_1,x_2\},
        \quad
        W_\theta'=\operatorname{span}\{y_1,y_2\}.
\]
Define a linear map from $W_\theta$ to $W_\theta'$ initially by
\[
        x_1\mapsto y_1,
        \quad
        x_2\mapsto y_2.
\]
We claim that this gives a well-defined isometry from $W_\theta$ onto $W_\theta'$.
Indeed, if $\lambda_1x_1+\lambda_2x_2=0$, then the equality of Gram matrices gives
\[
        \|\lambda_1y_1+
\lambda_2y_2\|^2
        =
        \|\lambda_1x_1+
\lambda_2x_2\|^2
        =0.
\]
Since the standard inner product on a real vector space is positive definite, $\lambda_1y_1+
\lambda_2y_2=0$.
Thus the map is well-defined.
The same Gram-matrix equality also shows that it preserves inner products.
Hence it is an isometry from $W_\theta$ onto $W_\theta'$.

Extend this isometry to an orthogonal map $Q_\theta$ on the whole $\theta$-eigenspace $\operatorname{im}E_\theta$, by choosing orthonormal bases of the orthogonal complements of $W_\theta$ and $W_\theta'$ and mapping one basis to the other.
Now define
\[
        Q:=\bigoplus_{\theta\in\Theta} Q_\theta
\]
with respect to the orthogonal decomposition of $\R^n$ into eigenspaces of $M$.
Then $Q$ is orthogonal.
Since $Q$ preserves each eigenspace of $M$, it commutes with $M$.
Moreover,
\[
        Q\bfe_a
        =Q\left(\sum_{\theta\in\Theta}E_\theta\bfe_a\right)
        =\sum_{\theta\in\Theta}Q_\theta E_\theta\bfe_a
        =\sum_{\theta\in\Theta}E_\theta\bfe_c
        =\bfe_c.
\]
Similarly $U\bfe_b=\bfe_d$.
This proves the lemma.
\end{proof}

The next lemma shows that the hypothesis of Lemma~\ref{lem:edge-intertwiner} is automatic for ordered edges in a $1$-walk-regular graph.

\begin{lemma}\label{lem:ordered-edges}
Let $X$ be a $1$-walk-regular graph with adjacency matrix $A$, and let
\[
        A=\sum_{\theta\in\Theta}\theta E_\theta
\]
be its spectral decomposition.  Let $(a,b)$ and $(c,d)$ be two ordered edges of $X$.  Then, for every $\theta\in\Theta$, the two pairs
\[
        (E_\theta\bfe_a, E_\theta\bfe_b)
        \quad\text{and}\quad
        (E_\theta\bfe_c,\ E_\theta\bfe_d)
\]
have the same Gram matrix.
\end{lemma}

\begin{proof}
By Lemma~\ref{lem:spectral-1wr}, for each $\theta\in\Theta$ there exist scalars $\alpha_\theta$ and $\beta_\theta$ such that
\[
        E_\theta\circ I=\alpha_\theta I,
        \qquad
        E_\theta\circ A=\beta_\theta A.
\]
Since $a\sim b$ and $c\sim d$, we have
\[
        \bfe_a^\top E_\theta\bfe_a
        =\bfe_b^\top E_\theta\bfe_b
        =\bfe_c^\top E_\theta\bfe_c
        =\bfe_d^\top E_\theta\bfe_d
        =\alpha_\theta,
\]
and
\[
        \bfe_a^\top E_\theta\bfe_b
        =\bfe_c^\top E_\theta\bfe_d
        =\beta_\theta.
\]
As $E_\theta$ is symmetric, the transpose entries are equal as well.  Hence both Gram matrices are
\[
        \begin{pmatrix}
        \alpha_\theta & \beta_\theta\\
        \beta_\theta  & \alpha_\theta
        \end{pmatrix}.
\]
The result follows.
\end{proof}


For an edge $e=\{a,b\}$ of a graph $X$, write $X\setminus e$ for the graph obtained from $X$ by deleting the edge $e$ but keeping all vertices.

We now arrive at the main result, a restatement of Theorem \ref{thm:intro-main}.
Sine every strongly regular graph is $1$-walk-regular, we answer Problem \ref{conj:GS} affirmatively.

\begin{theorem}
\label{thm:main}
Let $X$ be a $1$-walk-regular graph, and let $e$ and $f$ be two edges of $X$.
Then $X\setminus e$ and $X\setminus f$ are orthogonally degree-similar.
More precisely, if $e=\{a,b\}$ and $f=\{c,d\}$, then there exists an orthogonal matrix $Q$ such that
\[
        Q A(X\setminus e)Q^\top=A(X\setminus f)
\]
and
\[
        Q D(X\setminus e)Q^\top=D(X\setminus f).
\]
\end{theorem}

\begin{proof}
Let $A=A(X)$, and choose an ordering of the endpoints of $e$ and $f$, say
\[
        (a,b),
        \quad
        (c,d).
\]
By Lemmas~\ref{lem:edge-intertwiner} and \ref{lem:ordered-edges}, there exists an orthogonal matrix $Q$ such that
\[
        QA=AQ,
        \quad
        Q\bfe_a=\bfe_c,
        \qquad
        Q\bfe_b=\bfe_d.
\]
Since $Q$ is orthogonal and commutes with $A$, we have
\[
        QAQ^\top=A.
\]
Now put
\[
        B_e:=\bfe_a\bfe_b^\top+\bfe_b\bfe_a^\top,
        \quad
        B_f:=\bfe_c\bfe_d^\top+\bfe_d\bfe_c^\top.
\]
Then
\[
        A(X\setminus e)=A-B_e,
        \qquad
        A(X\setminus f)=A-B_f.
\]
Using $Q\bfe_a=\bfe_c$ and $Q\bfe_b=\bfe_d$, we get
\[
        QB_eQ^\top=B_f.
\]
Therefore
\[
        Q A(X\setminus e)Q^\top
        =Q(A-B_e)Q^\top
        =A-B_f
        =A(X\setminus f).
\]

It remains to check the degree matrices.  Since $X$ is $1$-walk-regular, it is regular; let its degree be $r$.  Deleting the edge $e=\{a,b\}$ decreases precisely the degrees of $a$ and $b$ by one.  Hence
\[
        D(X\setminus e)=rI-\bfe_a\bfe_a^\top-\bfe_b\bfe_b^\top.
\]
Similarly,
\[
        D(X\setminus f)=rI-\bfe_c\bfe_c^\top-\bfe_d\bfe_d^\top.
\]
Thus
\[
        Q D(X\setminus e)Q^\top
        =rI-\bfe_c\bfe_c^\top-\bfe_d\bfe_d^\top
        =D(X\setminus f).
\]
This proves the theorem.
\end{proof}


%

\begin{remark}
Theorem~\ref{thm:main} strengthens the known cospectrality results for edge-deleted strongly regular graphs.
Godsil, Sun, and Zhang \cite{GodsilSunZhang} proved that, for a strongly regular graph $X$, the graphs $X\setminus e$ and $X\setminus f$ are cospectral with respect to the adjacency matrix, the Laplacian, the signless Laplacian, and the normalized Laplacian.
Fan, Xing, Zhang, and Wang \cite{FanXingZhangWang} proved that the same graphs have the same $\mu$-polynomial.
The theorem above gives a simultaneous orthogonal similarity between the adjacency and degree matrices of the two edge-deleted graphs.  Consequently,
\[
        Q L_\mu(X\setminus e) Q^\top=L_\mu(X\setminus f),
\]
so the equality of $\mu$-polynomials follows immediately.
\end{remark}

\begin{remark}
The orthogonal matrix $U$ constructed in the proof need not be a permutation matrix.  Thus the result does not assert that $X\setminus e$ and $X\setminus f$ are isomorphic.  Indeed, examples in \cite{GodsilSunZhang} give large families of pairwise non-isomorphic edge-deleted graphs arising from strongly regular graphs, even though the present theorem shows that they are all orthogonally degree-similar.
\end{remark}

\begin{remark}
Let $G$ be a graph.
Two vertices $u$ and $v$ are called \emph{cospectral} for adjacency matrix  if
\[   \det(tI-G\setminus u)=\det(tI-G\setminus v),\]
where $G\setminus u$ denotes the subgraph of $G$ by deleting the vertex $u$ together with its incident edges.
The idea of cospectral vertices goes back to Schwenk \cite{Schwenk}, who used it to construct non-isomorphic cospectral trees.
The characterization of cospectral vertices was established in \cite{GodsilMcKay,GodsilSmith}.
In \cite{FanZhangWang}, the authors give a new characterization of cospectral vertices, that is, there exists an orthogonal intertwiner commuting with the adjacency matrix and sending one coordinate vector to the other.

Similarly, one can define the cospectral edges for adjacency matrix, that is, two edges $e$ and $f$ are called \emph{cospectral} for adjacency matrix  if
\[   \det(tI-G\setminus e)=\det(tI-G\setminus f).\]
If the two ordered pairs of projected endpoint vectors form the same Gram matrix in every eigenspace, then we have an orthogonal intertwiner commuting with the adjacency matrix and sending one pair of ordered endpoint vectors to the other; see Lemma \ref{lem:edge-intertwiner} and Lemma \ref{lem:ordered-edges}.
So, the cospectral edges may have potential use in cospectral graph construction.
\end{remark}

\end{document}